\documentclass{amsart}

\usepackage{amsmath}
\usepackage{amsthm}
\usepackage{hyperref}
\usepackage{amsfonts,graphics,amsthm,amsfonts,amscd,latexsym}
\usepackage{epsfig}
\usepackage{flafter}
\usepackage{mathtools}
\usepackage{comment}
\usepackage{stmaryrd}

\usepackage{mathabx,epsfig}

\hypersetup{
    colorlinks=true,    
    linkcolor=blue,          
    citecolor=blue,      
    filecolor=blue,      
    urlcolor=blue           
}
\usepackage{tikz}
\usetikzlibrary{graphs,positioning,arrows,shapes.misc,decorations.pathmorphing}

\tikzset{
    >=stealth,
    every picture/.style={thick},
    graphs/every graph/.style={empty nodes},
}

\tikzstyle{vertex}=[
    draw,
    circle,
    fill=black,
    inner sep=1pt,
    minimum width=5pt,
]
\usepackage[position=top]{subfig}
\usepackage{amssymb}
\usepackage{color}

\calclayout

\usetikzlibrary{decorations.pathmorphing}
\tikzstyle{printersafe}=[decoration={snake,amplitude=0pt}]

\newcommand{\rank}{\operatorname{rank}}

\newcommand{\supp}{\operatorname{supp}}

\newcommand{\alg}{\operatorname{alg}}

\newcommand{\pp}{\mathbb{P}}

\newcommand{\qq}{\mathbb{Q}}
\newcommand{\zz}{\mathbb{Z}}
\newcommand{\nn}{\mathbb{N}}
\newcommand{\rr}{\mathbb{R}}
\newcommand{\cc}{\mathbb{C}}
\newcommand{\kk}{\mathbb{K}}

\def\O#1.{\mathcal {O}_{#1}}			
\def\pr #1.{\mathbb P^{#1}}				
\def\af #1.{\mathbb A^{#1}}			
\def\ses#1.#2.#3.{0\to #1\to #2\to #3 \to 0}	
\def\xrar#1.{\xrightarrow{#1}}			
\def\K#1.{K_{#1}}						
\def\bA#1.{\mathbf{A}_{#1}}			
\def\bM#1.{\mathbf{M}_{#1}}				
\def\bL#1.{\mathbf{L}_{#1}}				
\def\bB#1.{\mathbf{B}_{#1}}				
\def\bK#1.{\mathbf{K}_{#1}}			
\def\subs#1.{_{#1}}					
\def\sups#1.{^{#1}}

\DeclareMathOperator{\reg}{reg}

\usepackage{tikz}
\usetikzlibrary{matrix,arrows,decorations.pathmorphing}

  \newtheorem{introthm}{Theorem}

  \newtheorem{theorem}{Theorem}[section]
  \newtheorem{lemma}[theorem]{Lemma}
  
  \newtheorem{corollary}[theorem]{Corollary}

  \newtheorem{definition}[theorem]{Definition}
  \newtheorem{example}[theorem]{Example}

  \newtheorem{question}[theorem]{Question}

\newtheorem{remark}[theorem]{Remark}

\theoremstyle{remark}

\numberwithin{equation}{section}

\usepackage[all]{xy}

\begin{document}

\title[Polarized endomorphisms of Fano varieties with complements]{Polarized endomorphisms of Fano varieties with complements}

\author[J.~Moraga]{Joaqu\'in Moraga}
\address{UCLA Mathematics Department, Box 951555, Los Angeles, CA 90095-1555, USA
}
\email{jmoraga@math.ucla.edu}

\author[J.I.~Y\'a\~nez]{Jos\'e Ignacio Y\'a\~nez}
\address{UCLA Mathematics Department, Box 951555, Los Angeles, CA 90095-1555, USA
}
\email{yanez@math.ucla.edu}

\author[W. Yeong]{Wern Yeong}
\address{UCLA Mathematics Department, Box 951555, Los Angeles, CA 90095-1555, USA
}
\email{wyyeong@math.ucla.edu}

\maketitle

\begin{abstract}
Let $X$ be a Fano type variety
and $(X,\Delta)$ be a log Calabi--Yau pair with $\Delta$ a Weil divisor.
If 
$(X,\Delta)$ admits a polarized endomorphism, then we show that 
$(X,\Delta)$ is a finite quotient of a toric pair.
Along the way, we prove that a klt Calabi--Yau pair $(X,\Delta)$ with standard coefficients
that admits a polarized endomorphism is the quotient of an abelian variety.
\end{abstract}

\setcounter{tocdepth}{1} 
\tableofcontents

\section{Introduction}

A {\em polarized endomorphism} on a variety $X$
is an endomorphism $f\colon X\rightarrow X$ for which
$f^*A\sim mA$ for some ample divisor $A$ on $X$ and $m\geq 2$. 
Although many varieties admit 
interesting endomorphisms, it is expected that varieties 
admitting polarized endomorphisms
have a much more restrictive geometry.
Two examples of such varieties
are toric varieties and abelian  varieties (see Example~\ref{ex1} and Example~\ref{ex3}). 
Furthermore, certain finite quotients of the aforementioned examples admit 
polarized endomorphisms (see Example~\ref{ex2} and Example~\ref{ex4}).
It is a folklore conjecture that a variety $X$ of klt type admitting a polarized endomorphism must be a finite quotient of a toric fibration over an abelian  variety. In the conjecture, it is essential to impose that $X$ has klt type singularities (see Example~\ref{ex:cone-over-elliptic}).
Further, we know that $\qq$-Gorenstein varieties admitting int-amplified endomorphisms are log canonical~\cite{BH93}.
In recent years, there has been a great amount of activity on this topic.
For instance, the folklore conjecture is known in several cases:
for surfaces~\cite{Nak02}, 
for smooth Fano $3$-folds~\cite{MZZ22},
for homogeneous varieties~\cite{PS89}, and 
for klt Calabi--Yau varieties~\cite{Meng-17}.
In~\cite{KT23}, Kawakami and Totaro proved that varieties admitting polarized endomorphisms satisfy Bott vanishing.

A {\em complement} of a variety $X$ 
is a boundary $\Delta$ for which $(X,\Delta)$ is log canonical and $K_X+\Delta\sim_\qq 0$.
More precisely, we say that $\Delta$ is an {\em $N$-complement} of $X$ if $N(K_X+\Delta)\sim 0$, which in particular implies that all coefficients of $\Delta$ are in $\frac{1}{N}\,\zz.$
Recently, there has been a vast activity in the so-called theory of complements (see, e.g.~\cite{Bir19,FM20}).
Motivated by the theory of complements on Fano varieties, we study polarized endomorphisms of Fano varieties that preserve a complement structure. The following is the main result of this article.

\begin{introthm}\label{introthm1}
Let $X$ be a Fano type variety and let $(X,\Delta)$ be a log Calabi--Yau pair with $\Delta$ reduced.
If the pair $(X,\Delta)$ admits a polarized endomorphism,
then $(X,\Delta)$ is a finite quotient of a toric log Calabi--Yau pair.
\end{introthm}

Theorem~\ref{introthm1} deals with Fano type varieties
with a reduced complement.
It is expected that every variety
that has a polarized endomorphism
admits a complement~\cite{BG17}.
Note that the conditions of Theorem~\ref{introthm1} are satisfied
when $\Delta$ is a $1$-complement
of the Fano type variety $X$.

In the setting of Theorem~\ref{introthm1},
the polarized endomorphism
$f\colon X\rightarrow X$ of the pair $(X,\Delta)$
satisfies \linebreak $f^{-1}(\Delta)=\Delta$
and $\Delta$ is said to be a {\em completely invariant divisor}.
Polarized endomorphisms fixing a divisor
have been extensively studied.
Completely invariant divisors are expected to 
be defined by low degree equations in $X$ (see, e.g.,~\cite{Hor17}).
In~\cite{HN11}, Hwang and Nakayama proved that on a Fano manifold $X$ of Picard rank one that is different from projective space, an endomorphism $f\colon X\rightarrow X$
that is \'etale outside a completely invariant divisor
must be an isomorphism.
In~\cite{MZ20}, Meng and Zhang proved that 
if $X$ is a smooth rationally connected variety that admits a polarized endomorphism $f\colon X\rightarrow X$ 
that is \'etale outside a completely invariant divisor $\Delta$, then $(X,\Delta)$ is a toric pair.
In~\cite{MZ23}, Meng and Zhong proved that if $X$ is a smooth rationally connected variety and $\Delta$ is a reduced divisor, then $(X,\Delta)$ is a toric pair if and only if $X$ admits an int-amplified endomorphism $f$ that is \'etale outside of $\Delta$. 
In summary, Theorem~\ref{introthm1} is already known when $X$ is a smooth Fano type variety.
It is expected that the techniques of~\cite{HN11,MZ20,MZ23}
can prove Theorem~\ref{introthm1} in the case that $X$ is a terminal Fano type variety.
However, we need to introduce some new ideas 
related to the Jordan property of fundamental groups
in order to settle the klt Fano type case.
In Remark~\ref{rem:comp}, we compare our main theorem with the previous results in the literature. 
In Example~\ref{ex2}, we show that the finite quotient in the Theorem~\ref{introthm1} is indeed necessary. 
This example is based on the family of examples by Koll\'ar and Xu~\cite{KX09arxiv} 
in which they show that a Fano variety of Picard number 1 with terminal singularities 
that admits a polarized endomorphism, might not be rational.

Our next theorem goes in a somewhat orthogonal direction.
It is about polarized endomorphisms
of klt Calabi--Yau pairs
with standard coefficients.

\begin{introthm}\label{introthm2}
Let $(X,\Delta)$ be a klt log Calabi--Yau pair with standard coefficients.
If $(X,\Delta)$ admits an int-amplified endomorphism, then $(X,\Delta)$ is a finite quotient of an abelian  variety.
\end{introthm} 

In~\cite{Yos21}, Yoshikawa proves a version of the previous theorem that requires $X$ being $\qq$-factorial but allows more general coefficients for $\Delta$ in the log Calabi--Yau klt pair $(X,\Delta)$.
Theorem~\ref{introthm2} applies to the setting of Fano varieties with complements
as explained in Example~\ref{ex4}.
The finite quotient in the statement of Theorem~\ref{introthm2} is indeed necessary (see Example~\ref{ex4}).
In summary, we can understand 
polarized endomorphisms 
of Fano varieties with complements
in two cases: 
when the complement is reduced, 
and when the complement is klt and has standard coefficients.
All the theorems in the introduction work for int-amplified endomorphisms (see Definition~\ref{def:amp}). 
We stated them in terms of polarized endomorphisms for the sake of exposition.

\subsection{Sketch of the proof}

In this subsection, we sketch the proof of Theorem~\ref{introthm1} and Theorem~\ref{introthm2}.
The proof of the first theorem uses techniques from the Jordan property for Fano varieties~\cite{BFMS20,Mor21}, automorphisms of log Calabi--Yau pairs~\cite{Hu18}, and singularities of $\mathbb{T}$-varieties~\cite{LLM19}.

Let $f\colon (X,\Delta)\rightarrow (X,\Delta)$ be a polarized endomorphism of the pair.
The proof consists of three steps.
First, we show that the fundamental group
$\pi_1^{\rm alg}(X\setminus \Delta)$ is virtually abelian  (see Theorem~\ref{thm:virt-ab-fun}).
Hence, there is a cover of log pairs
$g\colon (Y,\Delta_Y)\rightarrow (X,\Delta)$
for which the algebraic fundamental group
of $Y\setminus \Delta_Y$ is an abelian  group. 
In the setting of the theorem,
the covering variety $Y$ is still Fano type (see Lemma~\ref{lem:FT-cover}).
Secondly, we prove that some iteration of the polarized endomorphism of  $(X,\Delta)$
lifts to a polarized
 endomorphism of the log pair $(Y,\Delta_Y)$ (see Theorem~\ref{thm:lifting-int-amp}).
 For this step, we use the fact that $\pi_1(X^{\rm reg}\setminus \Delta)$
 is a finitely presented group.
 Thus, we have a commutative diagram where the horizontal arrows are polarized endomorphisms:
 \[
 \xymatrix{
(Y,\Delta_Y)\ar[d]_-{g}\ar[r]^-{f_Y} & (Y,\Delta_Y)\ar[d]^-{g} \\
(X, \Delta)\ar[r]^-{f^n} &  (X,\Delta). 
 }
 \]
Since $\pi_1^{\rm alg}(Y\setminus \Delta_Y)$ is abelian, we are reduced to the study of Galois polarized endomorphisms. 
By Theorem~\ref{thm:aut-CY}, we know that ${\rm Aut}(Y,\Delta_Y)$ is a finite extension of an algebraic torus.
In order to obtain the previous statement, we use the fact that $Y$ is Fano type, or at least rationally connected.
Then, we argue that the group $G\leqslant {\rm Aut}(Y,\Delta_Y)$ corresponding to the Galois endomorphism $f_Y$
is contained in the connected component of ${\rm Aut}(Y,\Delta_Y)^0$, which is an algebraic torus.
Finally, we turn to use the theory of $\mathbb{T}$-varieties.
If a maximal torus of ${\rm Aut}(Y,\Delta_Y)$ has rank less than $\dim Y$, then we argue that the quotient by $G$ will make the singularities of $Y$ worse. For instance, certain multiplicities must increase (see Lemma~\ref{lem:limit-mult}). The previous allows us to argue that either $(Y,\Delta_Y)$ is toric or the quotient by $G$ is not an endomorphism (see Theorem~\ref{thm:int-amplified-toric-quot-CY}). This finishes the proof.

In the case of Theorem~\ref{introthm2}, using the same technique as in the previous proof, we may lift the polarized endomorphism to the index one cover of $K_X+\Delta$.
As $\Delta$ has standard coefficients, the index one cover is a klt Calabi--Yau variety.
Then, we can apply~\cite[Theorem 1.9.(1)]{Meng-17} to deduce that $Y$ is a Q-abelian  variety, i.e., the finite quasi--\'etale quotient of an abelian  variety.
Then, by considering the Galois closure of the finite morphism $Y\rightarrow (X,\Delta)$, we conclude that $(X,\Delta)$ is the finite quotient of an abelian  variety. 
Clearly, in this case, the quotient may be ramified in codimension one, meaning that $\Delta$ may be non-trivial.

\subsection*{Acknowledgements}
The authors would like to thank Rohan Joshi, Sheng Meng, and Burt Totaro for many useful discussions and comments. 

\section{Preliminaries}

We work over the field $\cc$ of complex numbers.
The {\em  rank} of a group $G$ is the least number of generators of $G$.
We say that two $\rr$-divisors $A$ and $B$ on a normal variety $X$ are \textit{linearly equivalent} and write $A\sim B$ if $A-B$ is the divisor of a rational function on $X$. In particular, this implies that $A-B$ has integer coefficients. We write $(X;x)$ for the germ of an algebraic variety $X$ at a closed point $x\in X$.

In this section, we prove some preliminary results regarding Fano varieties, int-amplified endomorphisms, and torus actions.
For the singularities of the MMP, we refer the reader to~\cite{Kol13}, 
for Fano type varieties we refer the reader to~\cite{Mor22}, and 
for toric geometry, we refer the reader to~\cite{CLS11}.

\begin{definition}
{\em 
    An endomorphism $f\colon X\rightarrow X$ is a \emph{polarized endomorphism} if $f^*A\sim mA$ for some $m\geq 2$ and ample divisor $A$.
    An endomorphism $f\colon X\rightarrow X$ is said to be {\em int-amplified} if $f^*A-A$ is ample
    for certain ample Cartier divisor $A$ in $X$.
}
\end{definition}

\begin{definition}\label{def:amp}
{\em
    We say that a pair $(X,\Delta)$ 
    admits a \emph{polarized endomorphism} if there is a polarized endomorphism $f: X \rightarrow X$ such that $f^*(K_X+\Delta)=K_X+\Delta.$
    We say that a pair $(X,\Delta)$ admits a {\em int-amplified endomorphism} if there is an int-amplified endomorphism $f\colon X\rightarrow X$ such that
    $f^*(K_X+\Delta)=K_X+\Delta$.
    }
\end{definition}

In the previous definition, we are fixing a Weil divisor $K_X$ that represents the class of the canonical divisor and we ask the equality $f^*(K_X+\Delta)=K_X+\Delta$ to hold as an equality of divisors, not just of classes.

We recall the definition of the 
orbifold fundamental group
of a log pair.

\begin{definition}
{\em 
Let $(X,\Delta)$ be a pair with standard coefficients, i.e., coefficients of the form $1 - \frac{1}{m}$ for $m\in \nn$. Let $D_1,\ldots, D_l$ be the prime components of $\Delta$, and $\Delta = \sum_{i=1}^l \left(1 - \frac{1}{m_i}\right) D_i$. We define the {\em orbifold fundamental group} of the pair $(X,\Delta)$ as
\[
    \pi_1(X,\Delta) := \pi_1(X^{\mathrm{reg}}\setminus \supp(\Delta))/N
\]
where $N$ is the normal subgroup generated by the elements $\gamma_i^{m_i}$, where $\gamma_i$ is a loop around $D_i$.
}
\end{definition} 

The following lemma allows us to control 
whether certain finite cover of a Fano type variety is again of Fano type.
This lemma is well-known to the experts.

\begin{lemma}\label{lem:FT-cover}
Let $X$ be a Fano type variety.
Let $(X,\Delta)$ be a log Calabi--Yau pair with $K_X+\Delta\sim 0$.
Let $f\colon Y\rightarrow X$ be a finite cover 
with branched divisor contained in the support of $\Delta$.
Then, $Y$ is a Fano type variety.
\end{lemma}

\begin{proof}
Let $D$ be a boundary divisor on $X$ 
for which $(X,D)$ is klt
and $-(K_X+D)$ is big and nef.
For each prime divisor $P$ on $X$
let $m_P$ be the ramification
index of $f$ at $P$.
For $\epsilon>0$ small enough
the pair
\[
(X,(1-\epsilon)\Delta+\epsilon D)
\]
is klt 
and $-(K_X+(1-\epsilon)\Delta+\epsilon D)$ is big and nef. 
Set $\Gamma:=(1-\epsilon)\Delta+\epsilon D$.
Further, for every prime divisor $P$ on $X$ we may assume
\begin{equation}\label{eq:mult-div}
{\rm coeff}_P(\Gamma) > 1-\frac{1}{m_P}.
\end{equation} 
Write
\[
K_Y+\Gamma_Y = f^*(K_X+\Gamma).
\]
By inequality~\eqref{eq:mult-div}, the divisor $\Gamma_Y$ is effective.
Thus, $(Y,\Gamma_Y)$ is a klt log pair
and $-(K_Y+\Gamma_Y)$ is big and nef.
So $Y$ is of Fano type.
\end{proof}

\begin{definition}
 {\em
Let $(X,\Delta)$ be a log pair.
We say that a finite morphism
$f\colon (X,\Delta)\rightarrow (Y,\Delta_Y)$ is of
{\em endomorphism type} if there exists an isomorphism
of log pairs
$\phi\colon (Y,\Delta_Y)\rightarrow (X,\Delta)$ such that the composition
$\phi\circ f\colon (X,\Delta)\rightarrow (X,\Delta)$ is an endomorphism of log pairs.
If the composition is int-amplified, we say that $f\colon (X,\Delta)\rightarrow (Y,\Delta_Y)$ is of {\em int-amplified type}.
}
\end{definition}

\begin{lemma}\label{lem:int-amplified-on-boundary}
Let $(X,\Delta)$ be a log Calabi--Yau with $K_X+\Delta\sim 0$.
Let $f\colon (X,\Delta)\rightarrow (X,\Delta)$ be an int-amplified endomorphism.
For every prime component $S\subseteq \lfloor \Delta\rfloor$, there is a positive integer 
$i\geq 1$ for which $f^i(S)=S$ and
$f^i|_{S^\vee}\colon (S^\vee,\Delta_{S^\vee})\rightarrow (S^\vee,\Delta_{S^\vee})$ is an int-amplified endomorphism.
Here, the log pair $(S^\vee,\Delta_{S^\vee})$ is induced by adjunction of $(X,\Delta)$
to the normalization $S^\vee$ of $S$.
\end{lemma} 

\begin{proof}
The divisor $\lfloor \Delta\rfloor$
has finitely many components so 
and $f$ induces a permutation of these.
Thus, for some $i\geq 1$, we have 
$f^i(S)=S$.
Therefore, we have a commutative diagram:
\[
\xymatrix{
S^\vee\ar[d]^{\mu} \ar[r]^-{f^i|_{S^\vee}} & S^\vee\ar[d]^{\mu} \\
S\ar@{^{(}->}[d]^-{j} \ar[r] & S\ar@{^{(}->}[d]^-{j} \\
X \ar[r]^-{f^i} & X.\\
}
\]
In the previous diagram,
the morphism
$\mu\colon S^\vee\rightarrow S$ is the normalization morphism
and $j\colon S\hookrightarrow X$ is the inclusion of $S$.
Since $f^i$ is int-amplified, there is an ample divisor $A$ on $X$ for which
${f^i}^*A - A$ is ample.
Let $H:=\mu^*j^*A$. Note that $H$ is an ample divisor.
By the commutativity of the previous diagram, we have 
\[
(f^i|_{S^\vee})^*H - H = 
(f^i|_{S^\vee})^*\mu^*j^*A - \mu^*j^*A =
\mu^*j^*{f^i}^*A - \mu^*j^*A = 
\mu^*j^*({f^i}^*A - A).
\]
This implies that $f^i|_{S^\vee}$ is int-amplified.
By adjunction, we conclude that 
$f^i|_{S^\vee}\colon (S^\vee,\Delta_{S^\vee})\rightarrow (S^\vee,\Delta_{S^\vee})$ is a finite morphism of log pairs.
Henceforth, the morphism
$f^i|_{S^\vee} \colon (S^\vee,\Delta_{S^\vee})\rightarrow (S^\vee,\Delta_{S^\vee})$ is an int-amplified
endomorphism of the log Calabi--Yau pair $(S^\vee,\Delta_{S^\vee})$.
\end{proof}

The following lemma is well-known.
It states that on a normal variety,
only finitely many multiplicities
at closed points occur.

\begin{lemma}\label{lem:only-finite-mult}
Let $X$ be a normal variety.
Then $\mu(X;x)$ only takes finitely many possible values on closed points $x$ of $X$.
\end{lemma}

\begin{proof}
The multiplicity $\mu(X;x)$ at every closed point $x\in X$ is a positive integer.
On the other hand, for every $m$ the set 
$\{x\in X \mid \mu(X;x)\geq m\}$ is a closed subvariety of $X$.
By Noetherian induction, for some $m_0$ the set $\{x\in X \mid \mu(X;x)\geq m_0\}$ is empty.
\end{proof}

The following lemma states that,
under some mild conditions,
finite quotients of $\mathbb{G}_m$-singularities have arbitrarily large multiplicities.

\begin{lemma}\label{lem:limit-mult}
let $(X;x)$ be a normal singularity and 
$\mathbb{G}_m\leqslant {\rm Aut}(X;x)$ be a one-dimensional torus.
Assume that there exists a $\mathbb{G}_m$-equivariant resolution $Y\rightarrow X$ and a prime exceptional divisor $E\subset Y$ on which $\mathbb{G}_m$ acts as the identity.
For each $k\geq 1$, let $\mu_k\leqslant \mathbb{G}_m$ be the subgroup of $k$-roots of unity
and let $(X_k;x_k)$ be the quotient of $(X;x)$ by $\mu_k$. Then, we have 
\[
\lim_{k\rightarrow \infty} \mu(X_k;x_k)=\infty.
\]
\end{lemma}

\begin{proof}
The statement is local, so we may replace $X$ with a $\mathbb{G}_m$-equivariant affine neighborhood of $x\in X$.
By~\cite[Theorem 1]{Wat81}, we may find a semiprojective variety $Y$ and a $\qq$-ample divisor $D$ on $Y$ for which 
\[
X_0:={\rm Spec}\left(
\bigoplus_{m\geq 0}H^0(Y,\mathcal{O}_Y(mD) 
\right) 
\]
is equivariantly isomorphic to $X$.
Then, the variety $X_k$ is equivariantly isomorphic to
\[
X_{k,0}:={\rm Spec}\left( 
\bigoplus_{m\geq 0} 
H^0(Y,\mathcal{O}_Y(mkD) 
\right).
\]
For each $k$, the finite quotient 
$X_0\rightarrow X_{k,0}$ is induced by the graded inclusion of algebras.
Let $Z:={\rm Spec}(H^0(Y,\mathcal{O}_Y))$ so $Y$ is projective over $Z$.
The projective morphism $Y\rightarrow Z$ is not the identity; otherwise, no higher equivariant birational model of $X_0$ admits a divisor on which $\mathbb{G}_m$ acts as the identity.
Let $\widetilde{X}_0$ be the relative spectrum over $X$ of the sheaf 
$\bigoplus_{m\geq 0}\mathcal{O}_Y(mD)$.
Then, we have an equivariant projective birational morphism $\widetilde{X}_0\rightarrow X_0$ that contracts the only divisor on which $\mathbb{G}_m$ acts as the identity.
Let $x_0\in \widetilde{X}_0$ be the closed invariant point corresponding to $x$.
Let $x_{k,0}$ be the image of $x_0$ in $X_{k,0}$.
Let $z_0$ be the image of $x_{k,0}$ in $Z$.
Then, the maximal ideal of $x_{k,0}$ equals
\[
\mathfrak{m}_k:=m_{z_0}\oplus 
\bigoplus_{m\geq 1}  H^0(Y,\mathcal{O}_Y(kmD)),
\]
where $m_{z_0}$ is the maximal ideal of $z_0$ in $Z$.
For $k$ large enough, the homomorphism
\begin{equation}\label{eq:surj-1}
H^0(Y,\mathcal{O}_Y(mkD)) 
\rightarrow 
H^0(Y_{z_0},\mathcal{O}_{Y_{z_0}}
(mkD|_{Y_{z_0}}))
\end{equation}
is surjective
and the homomorphism
\begin{equation}\label{eq:surj-2}
H^0(Y,\mathcal{O}_Y(kmD))
\times
H^0(Y,\mathcal{O}_Y(k_1mD))
\rightarrow
H^0(Y,\mathcal{O}_Y((k+k_1)mD))
\end{equation}
is surjective for every $k_1\geq 1$ (see, e.g.,~\cite[Example 1.2.22]{Laz04a}).
By the surjectivity~\eqref{eq:surj-2}, we have 
\[
\mathfrak{m}_k/\mathfrak{m}_k^2 \simeq 
m_{z_0}/m_{z_0}^2 
\oplus H^0(Y,kmD)/m_{z_0}H^0(Y,kmD).
\]
By the surjection~\eqref{eq:surj-1}, we have 
\[
\dim_{\kk} 
(\mathfrak{m}_k/\mathfrak{m}_k^2)
\geq 
\dim_{\kk} 
(H^0(Y_{z_0},kmD|_{Y_{z_0}})).
\]
Thus, for $k$ large enough, we have 
\[
\mu(X_k;x_k) =
\mu(X_{k,0};x_{k,0}) 
\geq 
\dim_{\kk}(H^0(Y,kmD)/m_{z_0}H^0(Y,kmD))
\geq 
\dim_{\kk} 
(H^0(Y_{z_0},kmD|_{Y_{z_0}})).
\]
As $Y_{z_0}$ is a projective variety
and $mD|_{Y_{z_0}}$ an ample divisor,
the right side goes to infinity
when so does $k$.
This finishes the proof.
\end{proof} 

\begin{lemma}\label{lem:int-amplified-toric}
Let $(X,\Delta)$ be a $n$-dimensional
toric log Calabi--Yau pair.
Let $\mathbb{T}\leqslant {\rm Aut}(X,\Delta)$ be a maximal dimensional torus
and $G\leqslant \mathbb{T}$ be a finite group.
If the finite morphism
$(X,\Delta)\rightarrow (X/G,\Delta/G)$ is of int-amplified type, then the rank of $G$ equals $n$.
\end{lemma} 

\begin{proof}
We proceed by induction. The statement is clear in dimension one.
We can find a subtorus $G\leqslant \mathbb{T}_0 \leqslant \mathbb{T}$ such that $\mathbb{T}_0$ has the same rank as $G$.
Write $G\simeq \zz_{n_1}\oplus \dots \oplus \zz_{n_k}\leqslant \mathbb{T}_0$. 
For every $\ell\geq 1$, we set $G_\ell:=\zz_{n_1^\ell}\oplus \dots 
\oplus \zz_{n_k^\ell}\leqslant \mathbb{T}_0$.
It suffices to show that $\mathbb{T}_0$ equals $\mathbb{T}$. Assume this is not the case.
Let $\pi\colon (X,\Delta) \rightarrow (X/G,\Delta/G)$ be the finite quotient
and $\psi\colon (X/G,\Delta/G)\rightarrow (X,\Delta)$ be an isomorphism such that 
$\phi:=\psi\circ \pi\colon (X,\Delta)\rightarrow (X,\Delta)$ is an int-amplified endomorphism.
Thus, we have a commutative diagram
\[
\xymatrix{
(X,\Delta) \ar[d]_-{\pi}\ar[rd]^-{\phi} \\
(X/G,\Delta/G) \ar[r]^-{\psi} & (X,\Delta).
}
\]
Note that $\mathbb{T}/G$ induces a maximal torus action on $(X,\Delta)$ via $\phi$.
As all maximal tori of ${\rm Aut}(X,\Delta)$ are conjugate. Hence, up to replacing $\phi$, we may assume that the
action induced on $(X,\Delta)$ by $\mathbb{T}/G$ equals the $\mathbb{T}$-action.
Thus, for each $\ell\geq 1$, the quotient 
$(X/G_\ell,\Delta/G_\ell)$ 
is an int-amplified endomorphism 
and the quotient homomorphism
$\pi_\ell\colon (X,\Delta)\rightarrow (X/G_\ell,\Delta/G_\ell)$ is isomorphic
to $\phi^\ell$.

Let $S$ be a prime component of $\Delta$. 
Let $(S,\Delta_S)$ be the pair obtained by adjunction of $(X,\Delta)$ to $S$.
Then, the log Calabi--Yau pair $(S,\Delta_S)$ is a $(n-1)$-dimensional toric log Calabi--Yau pair.
For some $\ell\geq 2$, we have $\phi^\ell(S)=S$.
By Lemma~\ref{lem:int-amplified-on-boundary}, the quotient 
$(S,\Delta_S) \rightarrow (S/G_\ell,\Delta_S/G_\ell)$ is of int-amplified type.
By induction on the dimension, we conclude that $k\geq n-1$.
Therefore, the algebraic torus $\mathbb{T}_0$ has rank either $n-1$ or $n$. 
Furthermore, the restriction of $\mathbb{T}_0$ to each prime component of $\Delta$ has rank $n-1$.

Assume that $\mathbb{T}_0$ has rank $n-1$.
Let $\Sigma\subset N_\qq \simeq \qq^n$ be the fan of $X$ so 
$X\simeq X(\Sigma)$.
The algebraic torus $\mathbb{T}_0$ corresponds to a surjective homomorphism
$\rho \colon N_\qq \rightarrow N'_{\qq}\simeq \qq$.
Let $K$ be the kernel of $\rho$.
Let $p\colon \widetilde{X}\rightarrow X$ be a projective toric morphism such that $\widetilde{X}$ admits a quotient
$q\colon \widetilde{X}\rightarrow \pp^1$ 
for the $\mathbb{T}_0$-action.
The morphism $p$ corresponds to the fan refinement $\widetilde{\Sigma}$ of $\Sigma$ 
obtained by adding the cones $\sigma\cap K$ for each cone $\sigma \in \Sigma$.
On the other hand, the morphism $q$ corresponds to the projection $\pi_0$.
Let $Q$ be a $\mathbb{T}_0$-invariant prime divisor of $\widetilde{X}$ that is horizontal over $\pp^1$. 
We argue that $Q$ is a log canonical place of $(X,\Delta)$.
Indeed, let $(\widetilde{X},\widetilde{\Delta})$ be the log pull-back of $(X,\Delta)$ to $\widetilde{X}$. The restriction of 
$(\widetilde{X},\widetilde{\Delta})$ to a general fiber of $q$ is a toric sub-log Calabi--Yau pair.
Thus, all the torus invariant components, including $Q$, must appear with coefficient one.
The image of $Q$ on $X$ is either a divisor or an irreducible subvariety of codimension $2$.
Indeed, the prime divisor $Q$ corresponds to a ray $\rho_Q \in \widetilde{\Sigma}(1)$; 
the ray $\rho_Q$ is either on $\Sigma$
or is the intersection of a $2$-dimensional cone of $\Sigma$ with $K$.

By the previous paragraph, 
under the assumption that $\mathbb{T}_0$ has rank $n-1$, the variety $p(Q)$ has either codimension one or codimension two.
In both cases, we will obtain a contradiction.
Assume that $p(Q)$ is a divisor.
By construction, the restriction of the algebraic torus $\mathbb{T}_0$ to $p(Q)$ has 
rank $n-2$. This contradicts the fact that
the restriction of $\mathbb{T}_0$ to each component of $\Delta$ has rank $n-1$.
Otherwise, assume that the image of $p(Q)$
has codimension two.
Let $\mathbb{T}_1\leqslant \mathbb{T}_0$ be the one-dimensional subtorus that acts as the identity on $Q$ and $\mathbb{T}_1'$ be its complement torus in $\mathbb{T}_0$.
The torus $\mathbb{T}_1'$ acts faithfully on $p(Q)$.
Let $x$ be a general point of $p(Q)$.
Let $\mu_\ell \leqslant \mathbb{T}_1$ be the subgroup of $\ell$ roots of unity
and $x_{0,\ell}$ be the image of 
$x$ in $X/\mu_\ell$.
Let $x_\ell$ be the image of $x$ in $X/G_\ell$.
By construction, the germ $(X/G_\ell;x_\ell)$ is analytically isomorphic to $(X/\mu_\ell;x_{0,\ell})$.
Thus, 
by Lemma~\ref{lem:limit-mult}, we conclude that 
\[
\lim_{\ell \rightarrow \infty}\mu(X/G_\ell; x_\ell) =
\lim_{\ell \rightarrow \infty} 
\mu(X/\mu_\ell; x_{0,\ell}) =\infty.
\]
This contradicts Lemma~\ref{lem:only-finite-mult}.
We conclude that the rank of $\mathbb{T}_0$ equals $n$ and so does the rank of $G$. 
This finishes the proof.
\end{proof}

The following theorem is a well-known statement about the automorphism groups of log Calabi--Yau pairs (see, e.g.~\cite[Theorem 1.1]{Hu18}).

\begin{theorem}\label{thm:aut-CY}
Let $X$ be a Fano type variety and
$(X,\Delta)$ be a log Calabi--Yau pair.
Then, the automorphism group ${\rm Aut}(X,\Delta)$ is a finite extension of an algebraic torus.
\end{theorem}

\section{Lifting polarized endomorphisms to finite covers}

In this section, we develop techniques to lift polarized endomorphisms to finite covers of pairs.
Then, we turn to prove some results regarding the algebraic fundamental groups
of open Calabi--Yau varieties.

\begin{theorem}\label{thm:lifting-int-amp}
Let $X$ be a Fano type variety, and let $(X,\Delta)$ be a log Calabi--Yau pair, with $K_X + \Delta \sim 0$. Let $U:= X^{\rm{reg}}\setminus \Delta$.
Suppose that $(X,\Delta)$ admits an int-amplified endomorphism $f \colon (X,\Delta) \to (X,\Delta)$. 
Let $g\colon (Y,\Delta_Y) \to (X,\Delta)$ be a finite cover induced by a finite \'etale cover $g_0\colon U_0 \to U$.
Then there exists $n \gg 0$ such that the pullback of $f^n$ under the fiber product diagram 
\[
\xymatrix{
	(Y',\Delta_{Y'}) \ar[rr]^{\tilde{f}_{Y'}} \ar[d]_{g'} & & (Y,\Delta_Y) \ar[d]^{g}\\
	(X,\Delta) \ar[rr]_{f^n} & & (X,\Delta)
}
\]
is of int-amplified type.
\end{theorem}

\begin{proof}
Let $U_n := U\times_U U_0$ on the fiber product diagram 
\[
\xymatrix{
	U_n \ar[rr]^{f_n} \ar[d]_{g_n} & & U_0 \ar[d]^{g_0}\\
	U \ar[rr]_{f^n|_{U}} & & U
}
\]
Because $g_0$ is an \'etale finite cover of $U$,
the group $H_0:=(g_0)_*(\pi_1(U_0))$ is a subgroup of $\pi_1(U)$ of index $\deg(g_0)$.
The group $\pi_1(U)$ is finitely presented,
so there exist finitely many subgroups of index $\deg(g_0)$.
Therefore, for $n \gg 0$, $(g_0)_*(\pi_1(U_0)) = (g_n)_*(\pi_1(U_0)) = H_0$
and so $U_n$ is homeomorphic to $U_0$. 
Even more, as both $g_0$ and $g_n$ are analytic covers~\cite[Theorem 3.4]{DethloffGrauert94}, 
then there exists an isomorphism $h\colon U_0 \to U_n$ such that $g_n \circ h  = g_0$.
See also~\cite[Proposition 3.13]{GKP16arxiv}.

Consider the fiber product 
\[
\xymatrix{
	(Y',\Delta_{Y'}) \ar[rr]^{\tilde{f}_{Y'}} \ar[d]_{g'} & & (Y,\Delta_Y) \ar[d]^{g}\\
	(X,\Delta) \ar[rr]_{f^n} & & (X,\Delta)
}
\]
As both $Y' \xrightarrow{g'} X$ and $Y \xrightarrow{g} X$ are extensions of locally
biholomorphic coverings corresponding to the subgroup $H_0 \leq \pi_1(U)$, 
there exists an isomorphism $\tilde{h} \colon Y \to Y'$  with $g = g' \circ \tilde{h}$.
As both $Y\to X$ and $Y' \to X$ ramify along $\Delta$,
the isomorphism $\tilde{h}$ is in fact an isomorphism of pairs 
\[
    \tilde{h}\colon (Y,\Delta_Y) \to (Y',\Delta_{Y'}).
\]

Let $f' = \tilde{f}_{Y'} \circ \tilde{h} \colon (Y,\Delta_Y) \to (Y,\Delta_Y)$. 
We need to show that $f'$ is int-amplified.
Let $A$ be an ample divisor in $X$. 
Then $g^* ((f^n)^* A - A)$ is ample, as $g$ is finite and $f$ is int-amplified. 
Therefore, we have that 
\[
g^*((f^n)^* A) - g^* A = (f')^*(g^* A) - g^*A
\]
is ample, which implies that $f'$ is int-amplified.
Thus, the morphism $\tilde{f}^n$ is of int-amplified type.
\end{proof}

\begin{corollary}\label{coro:lifting-index-one}
Let $(X,\Delta)$ be a klt Calabi-Yau pair,
where $\Delta$ is assumed to have standard coefficients.
Let $Y \to (X,\Delta)$ be the index one cover of $K_X + \Delta$.
If $f\colon (X,\Delta) \to (X,\Delta)$
is an int-amplified endomorphism,
then there exists $n \gg 0$ such that
$f^n$ lifts to a finite morphism $f_Y\colon Y' \to Y$ of int-amplified type.
\end{corollary}

\begin{proof}
Following the proof of Theorem~\ref{thm:lifting-int-amp},
we obtain that for $n \gg 0$, finite morphism $g'$
in the fiber product diagram
\[
\xymatrix{
    Y' \ar[rr]^-{\tilde{f}_Y'} \ar[d]_-{g'} & & Y \ar[d]^-g\\
    (X,\Delta) \ar[rr]_-{f^n} & & (X,\Delta)
}
\]
is isomorphic to the index one cover $g$.
\end{proof}

The following result states that if $X$ is a Fano type variety, with $(X,\Delta)$ a log Calabi-Yau pair of index one, then there exists a finite cover $f\colon (Y,\Delta_Y) \to (X,\Delta)$ such that the algebraic fundamental group $\pi_1^{\rm{alg}}(Y\setminus \Delta_Y)$ is abelian.

\begin{theorem}\label{thm:virt-ab-fun}
Let $X$ be a Fano type variety of dimension $n$ variety and $(X,\Delta)$ be a log Calabi--Yau pair of index one.
Let $U:=X^{\rm reg}\setminus \Delta$.
Then, the group $\pi_1^{\rm alg}(U)$ is virtually abelian.
\end{theorem}

\begin{proof}
Let $P_1,\ldots ,P_s$ be the prime components of the divisor $\Delta$. 
Consider $\zz^s_{>0}$ as a directed poset, 
where $\vec{m} = (m_1,\ldots,m_s) \leq \vec{n} = (n_1,\ldots, n_s) $ 
if $m_i \mid n_i$ for every $1 \leq i \leq s$. 
Because $\zz^s_{>0}$ is countable, 
we can find an increasing sequence $(\vec{m}_i)_{i\in \nn} \subseteq \zz^d_{>0}$ 
that is cofinal, meaning that for any 
$\vec{n} \in \zz^s_{>0}$ there exists $\vec{m}_i$ 
with $\vec{n} \leq \vec{m}_i$. 

For each $\vec{n} = (n_1,\ldots, n_s) \in \zz^s_{>0}$, define $\Delta_{\vec{n}} := \sum \left( 1 - 1/n_j\right) P_j$.
Therefore, we can compute the following inverse limit by passing to the cofinal subset
\[ \varprojlim_{\vec{m}} \pi_1 (X,\Delta_{\vec{m}}) 
= \varprojlim_{i\in \nn} \pi_1 (X,\Delta_{\vec{m}_i}) .
\]

For each $\vec{m} \in \zz^s_{>0}$, there exists a surjective homomorphism 
$\pi_1(X\setminus \Delta) \to \pi_1(X,\Delta_{\vec{n}})$. 
Even more, by \cite[Theorem 1.2]{Bra20}, 
the groups $\pi_1^{\reg}(X,\Delta_{\vec{n}})$
are finite.
Hence, for each $i\in \nn$, there exists a normal subgroup 
$N_i$ of $\pi^{\reg}_1(X\setminus \Delta)$ of finite index such that 
\begin{equation}\label{eq:surj-pi-reg}
    \pi^{\reg}_1(X\setminus \Delta)/N_i \xrightarrow{\sim} \pi_1^{\reg}(X,\Delta_{\vec{m}_i})
\end{equation}
We show now that the sequence $(N_i)_{i\in \nn}$ is a 
cofinal subset of the poset of normal subgroups 
of $\pi_1(X\setminus \Delta)$ of finite index.

Let $N \leq \pi(X\setminus \Delta)$ be a normal subgroup of finite index.
We need to show that there exists $i \in \nn$ such that $N_i \leq N$.
Let
\[
    U_Y \to U
\]
be the finite \'etale Galois cover associated to $N$.
We extend this cover to a crepant finite Galois cover
\[
    (Y,\Delta_Y) \to (X,\Delta),
\]
possibly ramifying along the the divisor $\Delta$.
Thus, for some $\Delta_{\vec{n}} \leq \Delta$, there is a surjective
homomorphism
\[
    \pi_1(X,\Delta_{\vec{n}}) \to \pi_1(X^{\reg}\setminus \Delta)/N.
\]
Let $\vec{n} \mid \vec{m}_i$ for some $i\in \nn$. Then, there exists a surjective homomorphism
\[
    \pi_1(X,\Delta_{\vec{m}_i}) \to \pi_1(X,\Delta_{\vec{n}})
\]
From \eqref{eq:surj-pi-reg}, 
we obtain a surjective homomorphism 
\[
    \pi_1(X^{\reg}\setminus \Delta)/N_i \to \pi_1(X\setminus \Delta)/N
\]
impliying that $N_i \leq N$.

We conclude that 
\[
    \pi^{\alg}_1(U) = \varprojlim_{i\in \nn} 
    \pi_1(X^{\reg}\setminus \Delta)/N_i \xrightarrow{\sim} 
    \varprojlim_{i\in \nn} \pi_1 (X,\Delta_{\vec{m}_i})
\]

By \cite[Theorem 3]{BFMS20}, for each $\vec{m}_i$ we have an exact sequence
\[
    1 \to A_{i} \to \pi_1 (X,\Delta_{\vec{m}_i}) \to F_i \to 1
\]
where $A_i$ is abelian  and $|F_i| \leq C$, 
a constant that does not depend in $i$.
We can assume, by maybe passing to a subsequence, 
that $F_i = F$ for all $i\in \nn$.
Then, for $j\leq i$, we can induce a surjective map 
$A_i\to A_j$ 
such that the following diagram commutes, where the rows are exact
\[
\xymatrix{
	1 \ar[r] & A_i \ar[r] \ar[d] & \pi_1(X,\Delta_{\vec{m}_i}) \ar[r] \ar[d] & F \ar[r] \ar[d]^{\rotatebox{90}{$\sim$}} & 1\\
	1 \ar[r] & A_j \ar[r] & \pi_1(X,\Delta_{\vec{m}_j}) \ar[r] & F \ar[r] & 1
}
\]

Then, as all the groups involved are finite,
we obtain the exact sequence 
\[
    1 \to \varprojlim_{i\in \nn} A_{i} \to \pi^{\alg}_1(U) \to F \to 1
\]
and because each $A_i$ is abelian, the subgroup
$\displaystyle{\varprojlim_{i\in \nn} A_{i}}$ is abelian  of finite index
in $\pi^{\alg}_1(U)$,
so $\pi^{\alg}_1(U)$ is virtually abelian.
\end{proof}

If $\pi_1^{\rm{alg}}(Y\setminus \Delta_Y)$ is abelian, we can conclude that any finite \'etale cover of $Y\setminus \Delta_Y$ is Galois.

\begin{lemma}\label{lem:prof-comp-ab}
Let $G$ be a group.
Assume that its profinite completion $\widehat{G}$ is an abelian  group.
Then, every subgroup of finite index of $G$ is normal.
\end{lemma}

\begin{proof}
Let $H \leq G$ be a subgroup of finite index of $G$.
Consider the subgroup 
\[
K = \bigcap_{g\in G} gHg^{-1}
\]
of $G$.
By construction, $K$ is normal in $G$ and of finite index.
In particular, as $\widehat{G}$ is abelian, the quotient $G/K$ is abelian.
Then $H/K$ is a normal subgroup of $G/K$, which implies that $H$ is normal.
\end{proof}

\section{Varieties with torus actions}

In this section, we study varieties $X$
for which the quotient 
$X\rightarrow X/G$ with $G \leqslant {\rm Aut}(X)$ finite is of int-amplified type.

\begin{lemma}\label{lem:producing-boundary}
Let $(X,\Delta)$ be a log Calabi--Yau pair with $K_X+\Delta\sim 0$.
Let $\mathbb{G}_m^k \leqslant {\rm Aut}^0(X,\Delta)$ be an algebraic torus.
Let $G_i:=\bigoplus_{j=1}^k \zz/n_j^i \zz \leqslant \mathbb{G}_m^k$ where each $n_j\geq 2$.
If each finite morphism
$(X,\Delta)\rightarrow (X/G_i,\Delta/G_i)$ is of endomorphism type, then $\Delta \neq 0$.
\end{lemma}

\begin{proof}
Let $n$ be the dimension of $X$.
Assume that the statement does not hold, meaning that each 
finite morphism $X \rightarrow X/G_i$ is of endomorphism type and $\Delta=0$.
In particular, there is an upper bound $m_0$ such that for every $i$ and every closed point $x\in X/G_i$ we have
$\mu(X/G_i;x)\leq m_0$.
Let $\pi\colon \widetilde{X}\rightarrow X$ be an $\mathbb{G}_m^k$-equivariant projective birational morphism
for which $\widetilde{X}$ admits a quotient for the $\mathbb{G}_m^k$-action.
Thus, we have a $\mathbb{G}_m^k$-equivariant fibration $\widetilde{X}\rightarrow Z$ with general fiber a $(n-k)$-dimensional toric variety.
Let $\pi^*(K_X)=K_{\widetilde{X}}+\widetilde{\Delta} \sim 0$.
By construction, the sub-pair $(\widetilde{X},\widetilde{\Delta})$ is a sub log Calabi--Yau pair.
Further, the restriction of $(\widetilde{X},\widetilde{\Delta})$ to a general fiber of $\widetilde{X}\rightarrow Z$ is a sub log Calabi--Yau toric pair.
In particular, there is a prime component $S\subseteq \widetilde{\Delta}^{=1}$ that dominates $Z$.
Then, there is a one-dimensional subtorus $\mathbb{T}_0:=\mathbb{G}_m\leqslant \mathbb{G}_m^k$ that acts as the identity on $S$.
Let $\mathbb{T}_1$ be the split torus of $\mathbb{T}_0$ in $\mathbb{G}_m^k$.
For each $i$, we let $H_i$ be the restriction of $G_i$ to $\mathbb{T}_0$.
By construction, the sequence of groups $H_i$ is an infinite sequence of finite subgroups of $\mathbb{T}_0$.
Let $\pi(S)\subset X$ be its image on $S$ on $X$.
By~\cite[Theorem 10.1]{AH06}, the variety $\pi(S)\subset X$ contains a subvariety $S_0$ isomorphic to $\mathbb{G}_m^{k-1}$ on which $\mathbb{T}_1$ acts faithfully.
Let $x\in S_0$ be a general point.
Then, the pair $(X;x)$ is a normal singularity with a $\mathbb{T}_0$-action.
For each quotient $X\rightarrow X_i:=X/G_i$, we let $x_i$ be the image of $x$ on $X_i$.
Since $\mathbb{T}_1$ acts faithfully on $S_0$, we conclude that $(X_i;x_i)$ is locally isomorphic to $(X/H_i;x_i)$.
As $\mathbb{T}_0$ acts as the identity on $S$ from Lemma~\ref{lem:limit-mult}, 
we conclude that 
\[
\lim_{i\rightarrow \infty} \mu(X_i;x_i)=\infty.
\]
On the other hand, for each $i$, there is an isomorphism $\psi_i \colon X_i\rightarrow X$ so we have 
\[
\mu(X_i;x_i) = \mu(X;\psi_i(x_i)).
\]
Thus, there is a sequence of closed points on $X$ for which the multiplicity diverges.
This contradicts Lemma~\ref{lem:only-finite-mult}.
We conclude that $\Delta\neq 0$.
\end{proof} 

\begin{theorem}\label{thm:int-amplified-toric-quot-CY}
Let $(X,\Delta)$ be a $n$-dimensional log Calabi--Yau pair with $K_X+\Delta\sim 0$.
Let $G\leqslant \mathbb{T} \leqslant {\rm Aut}^0(X,\Delta)$ where $G$ is a finite group and $\mathbb{T}$ is an algebraic torus.
Assume that the finite morphism
$(X,\Delta) \rightarrow (X/G,\Delta/G)$ is of int-amplified type. Then, the following conditions are satisfied:
\begin{enumerate}
\item the algebraic torus $\mathbb{T}$ has rank $n$, 
\item the group $G$ has rank $n$, and 
\item the pair $(X,\Delta)$ is a log Calabi--Yau toric pair.
\end{enumerate}
\end{theorem} 

\begin{proof} 
The algebraic group ${\rm Aut}^0(X,\Delta)$ fits in an exact sequence
\[
1\rightarrow {\rm Aut}^0_L(X,\Delta) 
\rightarrow {\rm Aut}^0(X,\Delta) 
\rightarrow A(X,\Delta) \rightarrow 
1,
\]
where $A(X,\Delta)$ is a complex torus over the base field
and ${\rm Aut}^0_L(X,\Delta)$ is a linear algebraic group.
It is clear that $\mathbb{T} \leqslant {\rm Aut}^0_L(X,\Delta)$ and we may assume it is a maximal algebraic torus of this linear algebraic group.

There exists an isomorphism $\psi\colon (X/G,\Delta/G)\rightarrow (X,\Delta)$ of pairs making the following diagram commute:
\begin{equation}
\label{eq:first-com}
\xymatrix{
(X,\Delta)\ar[d]^-{/G} \ar[rd]^-{f} \\
(X/G,\Delta/G)\ar[r]^-{\psi} & (X,\Delta). 
}
\end{equation}
where $f$ is an int-amplified endomorphism of the log pair $(X,\Delta)$.
Since $\mathbb{T}/G \simeq \mathbb{T}$, the previous commutative diagram induces, via $f$, a maximal torus action on $(X,\Delta)$.
Thus, there is a finite subgroup $H \leqslant \mathbb{T}/G$ for which $(X,\Delta)\rightarrow (X/H,\Delta/H)$ is of int-amplified type. Furthermore, the group $H$ is isomorphic to $G$.

Set $G_1:=G$ and define $G_i$ to be the preimage of $G_1$ via $\pi_i$
in the following short exact sequence
\[
\xymatrix{ 
1 \ar[r] & 
G_{i-1} \ar[r] & 
\mathbb{T} \ar[r]^-{\pi_i} &
\mathbb{T}/G_{i-1} \ar[r] & 1.
}
\]
Thus, if $G\simeq \bigoplus_{j=1}^r \zz/n_j\zz$, then
$G_i\simeq \bigoplus_{j=1}^r \zz/n_j^i\zz$
for each $i$. 
Due to the commutative diagram~\eqref{eq:first-com}, for each $k$, we have a commutative
diagram as follows:
\begin{equation}
\label{eq:long-diag}
\xymatrix{
(X,\Delta)\ar[d]^-{/G_1} \ar[rd]^-{f} & &  \\ 
(X/G_1,\Delta/G_1)\ar[r]^-{\psi_1} \ar[d]^-{/G_{k-1}} & (X,\Delta)\ar[d]^-{/G_{k-1}}\ar[rd]^-{f^{k-1}} & \\ 
(X/G_k, \Delta/G_k)\ar[r]^-{\psi_k} & (X/G_1,\Delta/G_1)\ar[r]^-{\phi_{k-1}}  & (X,\Delta).\\
}
\end{equation}
In the previous diagram, all the vertical morphisms are finite quotients,
$\psi_1:=\psi$,
and $\phi_{k-1}:=\psi_1\circ \dots \circ \psi_{k-1}$.
Hence, due to the commutative diagram~\eqref{eq:long-diag}, for every $k$, we have a commutative diagram:
\begin{equation}
\label{eq:short-diag-k}
\xymatrix{ 
(X,\Delta) \ar[rd]^-{f^k} \ar[d]^-{/G_k} & \\ 
(X/G_k,\Delta/G_k) \ar[r]^-{\phi_k} & (X,\Delta).
} 
\end{equation} 
Up to passing to a sub-torus of $\mathbb{T}$,
we are in the situation of Lemma~\ref{lem:producing-boundary}.
Thus, we conclude that $\Delta\neq 0$.
Let $S\subseteq \Delta = \lfloor \Delta\rfloor$ be a prime component.
By Lemma~\ref{lem:int-amplified-on-boundary}, there exists $i$ for which the finite morphism
$(S^\vee,\Delta_{S^\vee})\rightarrow 
(S^\vee/G_i,\Delta_{S^\vee}/G_i)$
is of int-amplified type.
Observe that 
$(S^\vee,\Delta_{S^\vee})$ is a $(n-1)$-dimensional log Calabi--Yau pair
with $K_{S^\vee}+\Delta_{S^\vee}\sim 0$.
Further, we have 
$H_i\leqslant \mathbb{T}_S \leqslant {\rm Aut}^0(S^\vee,\Delta_{S^\vee})$
where $\mathbb{T}_S:=\mathbb{T}|_{S^\vee}$
and $H_i:=G_i|_{S^\vee}$.
Thus, by induction on the dimension, the following conditions are satisfied:
\begin{itemize}
\item[(i)] we have $\rank(\mathbb{T}_S)=n-1$,
\item[(ii)] the group $H_i$ has rank $n-1$, and 
\item[(iii)] the pair $(S^\vee,\Delta_{S^\vee})$ is a log Calabi--Yau toric pair.
\end{itemize} 
As the torus $\mathbb{T}_S$ is a homomorphic image of $\mathbb{T}$, we conclude that 
the rank of $\mathbb{T}$ is at least $n-1$.
From the previous argument, we deduce that every prime component of $\lfloor \Delta\rfloor$ is a, possibly non-normal, toric variety.
Further, the restriction of $\mathbb{T}$ to every prime component $S$ of $\lfloor \Delta\rfloor$ is a maximal torus of $S$.

We aim to show that $\mathbb{T}$ has rank $n$. This would show $(1)$.
For the sake of contradiction, assume that $\mathbb{T}$ has rank $n-1$.
Thus, the variety $X$ is a normal variety of complexity one (see, e.g.,~\cite{AH06}). 
Let $q\colon (\widetilde{X},\widetilde{\Delta})\rightarrow (X,\Delta)$ be a $\mathbb{T}$-equivariant projective birational morphism
such that $(\widetilde{X},\widetilde{\Delta})$ admits a $\mathbb{T}$-quotient to a smooth projective curve $C$.
By~\cite[Theorem 10.1]{AH06}, all the divisors contracted by $q\colon \widetilde{X}\rightarrow X$ are prime components of $\lfloor \widetilde{\Delta}\rfloor$ that are horizontal over $C$.
In particular, the log pair $(\widetilde{X},\widetilde{\Delta})$ is a log Calabi--Yau pair.
We write $\lfloor \widetilde{\Delta}\rfloor_{\rm hor}$ for the sum of the components of $\lfloor \widetilde{\Delta}\rfloor$ that are horizontal over $C$. 
On the other hand, since $\mathbb{T}$ acts fiberwise over $C$, then the restriction of $\mathbb{T}$ to any prime component of 
$\lfloor \widetilde{\Delta}\rfloor_{\rm hor}$
has rank at most $n-2$.
Thus, we conclude that $\lfloor \widetilde{\Delta}\rfloor_{\rm hor}$ is contracted by the projective birational morphism $\pi\colon \widetilde{X}\rightarrow X$.
Let $S$ be a prime component of
$\lfloor \widetilde{\Delta}\rfloor_{\rm hor}$.
Let $\mathbb{T}_0$ be the torus that acts as the identity on $S$ and $\mathbb{T}_1$ be its complement torus on $\mathbb{T}$.
The image of $S$ on $X$ is a subvariety of codimension $2$.
Let $x$ be the general point of $q(S)$.
For every $i$, we let $x_i$ to be the image of $x$ on $X/G_i$. For each $i$, there is an isomorphism $\psi\colon X/G_i\rightarrow X$
that induces an equality
\[
\mu(X/G_i;x_i) = \mu(X;\psi_i(x_i)).
\]
Let $\mu_i$ be the subgroup of $i$-roots of unity of $\mathbb{T}_0$.
Let $r_i\colon X\rightarrow X/\mu_i$.
Since $\mathbb{T}_1$ acts faithfully on $q(S)$, we have 
\[
\mu(X/G_i;x_i) = \mu(X/\mu_i;r_i(x)).
\]
By Lemma~\ref{lem:limit-mult} the left-hand side diverges.
Hence, we have a sequence of points $\psi_i(x_i))$ of $X$ for which 
the sequence $\mu(X;\psi_i(x_i))$ diverges.
This contradicts Lemma~\ref{lem:only-finite-mult}. Hence, we have concluded that $\mathbb{T}$ has rank $n$. This implies $(1)$ in the statement of the theorem.

Now, we turn to prove $(3)$.
We have $\mathbb{T}\leqslant {\rm Aut}^0(X)$ where $n$ is the dimension of $X$.
In particular, we have 
$\mathbb{T}\leqslant {\rm Aut}^0_L(X)$, so $\mathbb{T}$ has a regular effective action on $X$. Hence, $X$ is a toric variety.
As $\mathbb{T}\leqslant {\rm Aut}(X,\Delta)$, we conclude that $\Delta$ is a $\mathbb{T}$-invariant divisor, so the pair $(X,\Delta)$ is toric.
Thus, the log pair $(X,\Delta)$ is a toric log Calabi--Yau pair.
Finally, statement $(2)$ follows from $(3)$ and Lemma~\ref{lem:int-amplified-toric}.
\end{proof} 

\section{Proof of the Main Theorems}

In this section, we prove the main theorems of the article. First, we introduce a lemma regarding Galois endomorphisms. 

\begin{lemma}\label{lemma:aut_quotient}
Let $f\colon (X,\Delta) \to (X,\Delta)$ be an endomorphism
with $f^n$ Galois for all $n\geq 1$. Assume that $\mathrm{Aut}(X,\Delta)$ is a finite extension of an algebraic torus $\mathbb{T}$. Let $G_n \leqslant \mathrm{Aut}(X,\Delta)$ be the subgroup associated to $f^n$. Then, for some $n \geq 1$, we have $G_n \leqslant \mathbb{T}$.

\end{lemma}

\begin{proof}

For each $m,n\geq 1$, there is a short exact sequence 
\begin{equation}\label{eq:exact_seq_quotients}
1 \rightarrow G_n \rightarrow G_{n+m} \rightarrow G_m \rightarrow 1.
\end{equation}
By assumption, we have \[1\rightarrow \mathbb{T}
\rightarrow {\rm Aut}(X,\Delta) 
\xrightarrow{\pi} F \rightarrow 
1,\] where $F$ is a finite group.
Let $m\geq 1$ and take $n\gg 0$ such that $\pi(G_n) = \pi(G_{n+m})$.
Then $G_{n+m}/G_n$ is abelian, and by \eqref{eq:exact_seq_quotients} the group $G_m$ is also abelian.
Thus, $G_m$ is an abelian group for each $m\in \zz_{\geq 1}$.
\color{black}

Let $Z_m$ be the centralizer of $G_m$ in $\mathrm{Aut}(X,\Delta)$. As $G_m$ is abelian, the group $G_m$ is contained in $Z_m$.
We have the following chains of subgroups in $\mathrm{Aut}(X,\Delta)$:
\[
    G_1 \leq G_2 \leq \cdots \leq G_n \leq \cdots
\] 
\[
    Z_1 \geq Z_2 \geq \cdots \geq Z_n \geq \cdots
\]
We want to show that there exists a subgroup $Z_\infty \leq \mathrm{Aut}(X,\Delta)$ such that for some $N \geq 1$,
we have $Z_n = Z_\infty$ for all $n\geq N$.

It suffices to show that $Z_i \cap \mathbb{T}$ stabilizes for some $i\gg 0$. Indeed, $Z_i$ stabilizes if the images of $Z_i$ in $F$ and $Z_i \cap \mathbb{T}$ stabilize. The former follows from the finiteness of $F$, so it is enough to prove the latter.

Note that we have 
\[Z_i\cap \mathbb{T} = \bigcap\limits_{g\in G_i} \left(Z(g)\cap \mathbb{T} \right)\] for each $i$, where $Z(g)$ denotes the centralizer of $g$ in $\mathrm{Aut}(X,\Delta)$. 
Even more, $Z(g) \cap \mathbb{T} = Z(gt) \cap \mathbb{T}$ for all $t\in \mathbb{T}$. 
As $\mathrm{Aut}(X,\Delta)/\mathbb{T}$ is finite, there exist finitely many $g_1,\ldots, g_k \in \mathrm{Aut}(X,\Delta)$ such that for any $g\in \mathrm{Aut}(X,\Delta)$, $g = g_jt$ for some $1\leq j \leq k$ and $t\in \mathbb{T}$. 
Therefore, we can write 
\[
Z_i \cap \mathbb{T} = \bigcap\limits_{j\in I} (Z(g_j)\cap \mathbb{T})
\]
for some $I\subseteq \{1,\dots,k\}$.
Thus $Z_i\cap \mathbb{T}$ stabilize for $i\gg 0$.

Let $Z_\infty = \bigcap Z_i$, and let $Z_\infty^0$ be the connected 
component of $Z_\infty$.
Then, we have an exact sequence 
\begin{equation}\label{Zinfty-seq} 
\xymatrix{
1\ar[r] & Z_\infty^0 \ar[r] & Z_\infty \ar[r]^-{\pi_\infty}
& F_\infty \ar[r] & 1 
}
\end{equation} 
where $F_\infty$ is a finite group,
being $Z_\infty$ a linear algebraic group.
Using~\eqref{Zinfty-seq}, for any $G_n$ we obtain
an exact sequence
\begin{equation}\label{eq:exact_seq_quotientgroup}
    1 \to Z_\infty^0/ (Z_\infty^0 \cap G_n) \to Z_\infty/G_n \to F_\infty/\pi(G_n) \to 1.
\end{equation}
For all $i\gg 0$, we have $G_i\leq Z_\infty \leq N_i$, where $N_i$ is the normalizer of $G_i$ in $\mathrm{Aut}(X,\Delta)$. 
For every integers $m>n$, we have a commutative diagram 
\begin{equation}\label{eq:quot}
\xymatrix{
(X,\Delta)\ar[dd]_-{/G_n}\ar[rrdd]^-{/G_m} & \\
\\
(X/G_n,\Delta/G_n)\ar[rr]^-{/(G_m/G_n)} & & (X/G_m,\Delta/G_m). 
}
\end{equation} 
The endomorphism $f^{m-n}$ is isomorphic
to the quotient in the bottom of the commutative diagram~\eqref{eq:quot}.
Thus, it suffices to show that $G_m/G_n$ is contained in the connected component of ${\rm Aut}(X/G_n,\Delta/G_n)$.

Take $m > n \gg 0$ such that the images of $G_m$ and $G_n$ in $F_\infty$ via $\pi_\infty$ agree. For each $m \geq 1$, there exists a homomorphism $N_n/G_n \to \mathrm{Aut}(X/G_n,\Delta/G_n)$. Therefore, we obtain a commutative diagram:
\[
\xymatrix{
    N_n/G_n \ar[r] & \mathrm{Aut}(X/G_n,\Delta/G_n)\\
    Z_\infty/G_n \ar@{^{(}->}[u] \ar[ur] \\
    Z_\infty^\circ/ (Z_\infty^\circ \cap G_n) \ar@{^{(}->}[u] \ar[r] & \mathrm{Aut}^\circ(X/G_n,\Delta/G_n) \ar@{^{(}->}[uu]
}
\]
Since $Z_\infty^0$ is a connected linear algebraic group its image in ${\rm Aut}(X/G_m,\Delta/G_m)$ lies in the connected component.
On the other hand, due to~\eqref{eq:exact_seq_quotientgroup}, 
we have a commutative diagram
\[
    \xymatrix{
        1 \ar[r] & Z_\infty^0/(Z_\infty^0 \cap G_n) \ar[r] & Z_\infty/G_n \ar[r] & F_\infty/ \pi_\infty(G_n) \ar[r] & 1\\
        1 \ar[r] & (Z^0_\infty \cap G_m)/(Z_\infty^0 \cap G_n) \ar[u] \ar[r] & G_m/G_n \ar[r] \ar[u] & \pi_\infty(G_m)/\pi_\infty(G_n)\ar[u] \ar[r] & 1
    }
\]
Since $\pi_\infty(G_m)/\pi_\infty(G_n)=1$, 
we conclude that the image of $G_m/G_n$ in ${\rm Aut}(X/G_m,\Delta/G_m)$ is contained in the image of $Z_\infty^0/(Z_\infty^0\cap G_m)$ in ${\rm Aut}(X/G_m,\Delta/G_m)$.
Hence, the image of $G_m/G_n$ is contained in the connected component of ${\rm Aut}(X/G_n,\Delta/G_n)$. This finishes the proof of the lemma.
\end{proof}

Now, we turn to prove the main theorem of the article.

\begin{proof}[Proof of Theorem~\ref{introthm1}]
Let $p_0\colon X_0\rightarrow X$ be the index one cover of $K_X+\Delta\sim_\qq 0$.
As $K_X+\Delta$ is a Weil divisor, the finite morphism
$p_0\colon X_0\rightarrow X$ is unramified in codimension one.
Let $p_0^*(K_X+\Delta)=K_{X_0}+\Delta_{X_0}$.
By Lemma~\ref{lem:FT-cover}, we conclude that
$X_0$ is a Fano type variety
and $\Delta_{X_0}$ is a $1$-complement of $X_0$.
By Theorem~\ref{thm:lifting-int-amp}, for some $m\geq 1$, we have a commutative diagram:
\[
\xymatrix{
(X_0,\Delta_0)\ar[d]^-{p_0} \ar[r]^-{f_0} & (X_0,\Delta_0)\ar[d]^-{p_0} \\ 
(X,\Delta) \ar[r]^-{f^m} & (X,\Delta).
}
\]
In the previous commutative diagram the
endomorphism $f_0$ is a polarized
endomorphism of the pair $(X_0,\Delta_0)$.
Replacing $(X,\Delta)$ with $(X_0,\Delta_0)$, 
we may assume that $\Delta$ is a $1$-complement
of the Fano type variety $X$.

From now on, we assume that $X$ is a Fano type variety
and $(X,\Delta)$ is a log Calabi--Yau pair of index one.
Let 
$f\colon (X,\Delta)\rightarrow (X,\Delta)$ be an int-amplified endomorphism.
Let $U=X^{\rm reg}\setminus \Delta$.
By Lemma~\ref{thm:virt-ab-fun}, we know that $\pi_1^{\rm alg}(U)$ is a virtually abelian  group.
Let $N\leqslant \pi_1^{\rm alg}(U)$ be a normal abelian  subgroup of finite index.
Let $U'\rightarrow U$ be the corresponding finite cover. 
Let $g\colon (Y,\Delta_Y)\rightarrow (X,\Delta)$ be the induced finite cover of log pairs. 
Set $U_Y=Y^{\rm reg}\setminus \Delta_Y$.
Then, we have that $\pi_1^{\rm alg}(U_Y)$ is an abelian  group.
Indeed, $U_Y$ is smooth and $U'$ is a big open subset of $U_Y$.
In particular, $\pi_1^{\rm alg}(Y,\Delta_Y)$ is an abelian  group.
By Theorem~\ref{thm:lifting-int-amp}, for some $n\geq 1$, we have a commutative diagram:
\[
\xymatrix{
(Y,\Delta_Y)\ar[d]^-{g}\ar[r]^-{f_Y} & (Y,\Delta_Y)\ar[d]^-{g} \\ 
(X,\Delta)\ar[r]^-{f^n} & (X,\Delta).
}
\]
In the previous diagram, $f_Y$ is an int-amplified endomorphism for the log pair $(Y,\Delta_Y)$.

The finite morphism $f_Y\colon (Y,\Delta_Y)\rightarrow (Y,\Delta_Y)$
corresponds to a subgroup
of finite index
of $\pi_1(Y,\Delta_Y)$.
The profinite completion of
$\pi_1(Y,\Delta_Y)$ is
$\pi_1^{\rm alg}(Y,\Delta_Y)$.
Thus, by Lemma~\ref{lem:prof-comp-ab} 
the finite morphism $f_Y\colon (Y,\Delta_Y)\rightarrow (Y,\Delta_Y)$ corresponds to a normal subgroup of finite index of $\pi_1(Y,\Delta_Y)$.
Henceforth, the finite morphism
$f_Y\colon (Y,\Delta_Y)\rightarrow (Y,\Delta_Y)$ is Galois.
Therefore, there exists a finite group $G\leqslant {\rm Aut}(Y,\Delta_Y)$ and an isomorphism of log pairs
$\psi \colon (Y/G,\Delta_Y/G)\rightarrow (Y,\Delta_Y)$ making the following diagram commutative:
\[
\xymatrix{
(Y,\Delta_Y)\ar[rd]_-{f_Y}\ar[r]^-{/G} & (Y/G,\Delta_Y/G) \ar[d]^-{\psi} \\
& (Y,\Delta_Y).
}
\]
In particular, the quotient
$(Y,\Delta_Y)\rightarrow (Y/G,\Delta_Y/G)$ is of int-amplified type.
By Lemma~\ref{lem:FT-cover}, the variety $Y$ is of Fano type. Hence,
by
Theorem~\ref{thm:aut-CY}, we know that ${\rm Aut}(Y,\Delta_Y)$ is a finite extension of an algebraic torus.
By Lemma~\ref{lemma:aut_quotient}, up to replacing $f$ with an interation $f^n$, we may assume that $G$ is a finite subgroup of a maximal algebraic torus of ${\rm Aut}(Y,\Delta_Y)$.
By Theorem~\ref{thm:int-amplified-toric-quot-CY}, we conclude that $(Y,\Delta_Y)$ is a log Calabi--Yau toric pair.

We conclude that the pair $(X,\Delta)$ is the finite quotient of a toric log Calabi--Yau pair $(Y,\Delta_Y)$. This finishes the proof.
\end{proof}

\begin{proof}[Proof of Theorem~\ref{introthm2}]
Let $(X,\Delta)$ be a klt log Calabi--Yau pair with standard coefficients.
Let $f\colon Y\rightarrow X$ be the index one cover of $K_X+\Delta$,
so we have $f^*(K_X+\Delta)=K_Y$.
Then, $Y$ is a klt Calabi--Yau variety. 
By Corollary~\ref{coro:lifting-index-one}, we have a commutative diagram 
\[
\xymatrix{
Y\ar[d]_-{f}\ar[r]^-{g_Y} & Y\ar[d]^-{f} \\ 
(X,\Delta)\ar[r]^-{g^n} & (X,\Delta)  
}
\]
where $n\geq 1$ and $g_Y$ is an int-amplified endomorphism.
By~\cite[Theorem 1.9.(1)]{Meng-17}, we conclude that $Y$ is a finite quotient, unramified in codimension one, of an abelian  variety $A$.
The finite morphism 
$h\colon A \rightarrow (X,\Delta)$
corresponds to a subgroup $H\leqslant \pi_1^{\rm reg}(X,\Delta)$ of finite index.
Hence, we may find a subgroup $N\leqslant H \leqslant \pi_1^{\rm reg}(X,\Delta)$ of finite index that is normal in $\pi_1^{\rm reg}(X,\Delta)$.
Let $h'\colon A'\rightarrow (X,\Delta)$ be the finite Galois morphism induced by $N$.
It suffices to show that $A'$ is an abelian  variety.
We have ${h'}^*(K_X+\Delta)=K_{A'}+E$ where $E$ is an effective divisor.
On the other hand, we have a commutative diagram of finite morphisms:
\[
\xymatrix{
A'\ar[d]_-{\phi}\ar[rd]^-{h'} & \\
A \ar[r]^-{h} & (X,\Delta).\\
}
\]
By construction, we know that 
$h^*(K_X+\Delta)=K_A$, so 
\[
K_{A'}+E =
\phi^*(K_X+\Delta)=\phi^*K_A = K_{A'}-R, 
\]
where $R$ is the ramification divisor of $\phi$.
We conclude that $E+R=0$, where both $E$ and $R$ are effective, so $E=R=0$ and $\phi$ is unramified in codimension one.
Since $A$ is smooth, this turns out to imply that $\phi$ is an \'etale morphism.
A finite \'etale morphism to an abelian  variety is an isogeny so $A'$ is an abelian  variety.
We conclude that $(X,\Delta)$ is a finite quotient of an abelian  variety.
\end{proof} 

\section{Examples and Questions}

In this section, we provide some examples related to the main theorems of the paper and questions for further research.
The first example is the most well-known in this direction.

\begin{example}\label{ex1}
{\em  
Consider the endomorphism
\begin{align*}
f_m: (\pp^n,H_0+\dots+H_n) 
&\rightarrow 
(\pp^n,H_0+\dots+H_n)\\ 
f_m([x_0:\dots:x_n])&=
[x_0^m:\dots:x_n^m].
\end{align*} 
Then, $f_m$ is a polarized endomorphism as 
$f_m^*H_i=mH_i$ for every hyperplane coordinate. 
By Riemann-Hurwitz, we have 
\[
f_m^*\left(K_{\pp^n}+\sum_{i=0}^n\left(1-\frac{1}{m}\right)H_i\right) = K_{\pp^n}. 
\]
Thus, we conclude that 
\[
f_m^*\left( K_{\pp^n}+\sum_{i=0}^n H_i\right) = K_{\pp^n}+\sum_{i=0}^n H_i.
\]
Thus, the log Calabi--Yau pair $(\pp^n,H_0+\dots+H_n)$ admits a polarized endomorphism.
This gives an example of Theorem~\ref{introthm1}.

The previous example can be generalized in the following way.
For every $n$-dimensional projective toric variety $X$, we consider the log Calabi-Yau pair $(X,\Delta)$ where $\Delta$ is the reduced torus invariant divisor.
Consider the polarized endomorphism 
\[
f_m\colon X \rightarrow X
\]
corresponding to the inclusion of lattices $m\zz^n \subseteq \zz^n$
in the fan $\Sigma$ of $X$. 
It is clear that $f_m^*\Delta=m\Delta$
so we have $f_m^*(K_X+\Delta)=K_X+\Delta$ as above.
If $X$ is projective, then it admits a torus invariant ample divisor $H$, for which $f_m^*H=mH$.
The previous statement implies that $f_m\colon (X,\Delta)\rightarrow (X,\Delta)$ is a polarized endomorphism.
}
\end{example} 

The second example shows that the quotient in Theorem~\ref{introthm1} is necessary.

\begin{example}\label{ex2}
{\em 
Consider the pair $(\pp^n,\Delta')$ where $\Delta'$ is the reduced torus invariant boundary divisor and $n\geq 3$, which admits a polarized endomorphism $f_m$ as in Example~\ref{ex1}.
Let $S_{n+1}$ be acting on $\pp^n$ by permutations of the components.
Let $(X,\Delta):=(\pp^n/S_{n+1},\Delta'/S_{n+1})$.
The variety $X$ is a klt Fano variety
and 
$K_X+\Delta\sim 0$.
Further, $X$ is not a toric variety as
$\pi_1(X^{\reg}\setminus \Delta)$ surjects onto $S_{n+1}$, 
while the smooth locus of a toric variety contains $\mathbb{G}_m^n$ so its fundamental group is abelian.

The polarized endomorphism $f_m$ descends to an endomorphism $g_m$ of $(X,\Delta)$ as follows:
\[
\xymatrix{
(\pp^n,\Delta')\ar[d]^-{/S_{n+1}} \ar[r]^-{f_m} & 
(\pp^n,\Delta')\ar[d]^-{/S_{n+1}} \\ 
(X,\Delta) \ar[r]^-{g_m} &(X,\Delta).
}
\]
By construction, we have $g_m^*(K_X+\Delta)=K_X+\Delta.$ 
Since $q^*\Delta=\Delta'$ and $f_m^*\Delta'=m\Delta',$ we have
\[ g_m^*\Delta=q_*f_m^*q^*\Delta=q_*f_m^*\Delta'=q_*(m\Delta')=m\Delta. \]
Hence, $g_m$ is a polarized endomorphism.

This example is based on~\cite{KX09arxiv}. In fact, as is referred in that article, in~\cite{Saltman84} one can find an example of a nonabelian subgroup $G\leq S_{|G|+1}$ of order $p^9$, for some prime number $p$, for which the quotient $\pp^{n}/G$ is not rational, and hence not toric.
}
\end{example} 

\begin{remark}\label{rem:comp}
{\em 
In~\cite[Corollary 1.4]{MZ19}, Meng and Zhang show that 
if $X$ is a smooth rationally connected 
and $D$ a reduced divisor that is 
$f^{-1}$-invariant for a polarized endomorphism $f$,
where $f|_{X\setminus D}$ is \'etale ,
then $(X,D)$ is a toric pair. As in Example~\ref{ex2}, 
for $n\geq 3$,
$(X,D)$ is not a toric pair and it is smooth in codimension 2,
with at worst canonical singularities,
we see that either $f|_{X\setminus D}$ must be \'etale, or the smoothness condition is essential for \cite[Corollary 1.4]{MZ19}.
}
\end{remark}

The third example is related to Theorem~\ref{introthm2}.

\begin{example}\label{ex3}
{\em 
Let $A$ be an abelian  variety of dimension $n$ with corresponding rank $2n$ lattice $\Lambda\subset \cc^{n}.$ 
For each integer $m$, the map $g_m\colon \cc^n\rightarrow \cc^n$
given by $z\mapsto mz$ preserves the lattice, so it induces an endomorphism
\begin{equation}\label{eq:abelian -polarized-endomorphism}
\begin{split}
f_m: A 
&\rightarrow 
A\\ 
f_m(a)&=
ma
\end{split} 
\end{equation}
which is polarized when $m\neq0$ since $f_m^*H\sim m^2H$ for every symmetric ample divisor $H$. 
}
\end{example}

Finally, the fourth example shows that the finite quotient in Theorem~\ref{introthm2} is necessary.

\begin{example}\label{ex4}
{\em 
Let $E$ be an elliptic curve with corresponding lattice $\Lambda\subset \cc.$ Consider its quotient by the involution map $f_{-1}$ from (\ref{eq:abelian -polarized-endomorphism}), which gives
\[q:E \xrightarrow{/f_{-1}} (\pp_1,\Delta),\]
where $\Delta=\frac{1}{2}\left\{0\right\}+\frac{1}{2}\left\{1\right\}+\frac{1}{2}\left\{\lambda\right\}+\frac{1}{2}\left\{\infty\right\}.$
The polarized endomorphism $f_m$ from (\ref{eq:abelian -polarized-endomorphism}) descends to an endomorphism $g_m$ of $(\pp^1,\Delta)$ as follows:
\[
\xymatrix{
E\ar[d]^-{q} \ar[r]^-{f_m} & 
E\ar[d]^-{q} \\ 
(\pp^1,\Delta) \ar[r]^-{g_m} &(\pp^1,\Delta).
}
\] 
From the commutativity of the diagram, it follows that 
$g_m^*(K_{\pp^1}+\Delta)=K_{\pp^1}+\Delta$ and $g_m^*\Delta \sim m^2 \Delta$, so 
$g_m$ is a polarized endomorphism of the pair $(\pp^1,\Delta)$.
}
\end{example} 

The following example shows that log canonical varieties
with polarized endomorphisms may not be finite quotients
of toric fibrations over abelian varieties. Thus, in the folklore conjecture, it is necessary to impose that the variety has klt singularities.

\begin{example}\label{ex:cone-over-elliptic}
{\em 
Consider the projectivized cone $X$ over an elliptic curve $E$, which admits polarized endomorphisms induced by polarized endomorphisms on $E$ as in Example~\ref{ex3}.
At the cone vertex, $X$ has a singularity that is log canonical but not klt type. 
However, $X$ is not a finite quotient of a toric fibration over an abelian variety since such varieties must be klt type. 
}
\end{example}

Two natural questions emanate from this article.

\begin{question}
Let $X$ be a Fano variety and $\Delta$ be a complement. 
Assume that $(X,\Delta)$ admits a polarized endomorphisms.
Is $(X,\Delta)$ a finite quotient of a log Calabi--Yau toric fibration over an abelian variety? 
\end{question}

Our techniques so far can only prove
the previous statement when $K_X+\Delta\sim_\qq 0$
with $\Delta$ reduced, or
when $(X,\Delta)$ is klt log Calabi--Yau with standard coefficients.
We expect that the previous question has a positive answer, but it goes beyond the scope of this article. Some new ideas are required to conclude. 
This question will be addressed in an upcoming paper by the authors.
Let's say that a polarized endomorphism $f\colon X \rightarrow X$ admits a {\em complement} if there exists a complement $(X,\Delta)$ for which $f^*(K_X+\Delta)=K_X+\Delta$.
Not every polarized endomorphism admits a complement.
The hardest task seems to be the following:

\begin{question}
Let $X$ be a normal projective variety admitting a polarized endomorphism $f$.
Can we construct a polarized endomorphism on $X$ that admits a complement?  
\end{question}

We expect that a positive answer to the previous question would settle the main conjecture on the topic. 
However, nowadays, it is not even clear how to construct more polarized endomorphisms from a given one.

\bibliographystyle{habbvr}
\bibliography{references}

\begin{thebibliography}{10}
\expandafter\ifx\csname url\endcsname\relax
  \def\url#1{\texttt{#1}}\fi
\expandafter\ifx\csname doi\endcsname\relax
  \def\doi#1{\burlalt{doi:#1}{http://dx.doi.org/#1}}\fi
\expandafter\ifx\csname urlprefix\endcsname\relax\def\urlprefix{URL }\fi
\expandafter\ifx\csname href\endcsname\relax
  \def\href#1#2{#2}\fi
\expandafter\ifx\csname burlalt\endcsname\relax
  \def\burlalt#1#2{\href{#2}{#1}}\fi

\bibitem{AH06}
K.~Altmann and J.~Hausen.
\newblock Polyhedral divisors and algebraic torus actions.
\newblock {\em Math. Ann.}, 334(3):557--607, 2006.
\newblock \doi{10.1007/s00208-005-0705-8}.

\bibitem{Bir19}
C.~Birkar.
\newblock Anti-pluricanonical systems on {F}ano varieties.
\newblock {\em Ann. of Math. (2)}, 190(2):345--463, 2019.
\newblock \doi{10.4007/annals.2019.190.2.1}.

\bibitem{Bra20}
L.~Braun.
\newblock The local fundamental group of a {K}awamata log terminal singularity is finite.
\newblock {\em Invent. Math.}, 226(3):845--896, 2021.
\newblock \doi{10.1007/s00222-021-01062-0}.

\bibitem{BFMS20}
L.~Braun, S.~Filipazzi, J.~Moraga, and R.~Svaldi.
\newblock The {J}ordan property for local fundamental groups.
\newblock {\em Geom. Topol.}, 26(1):283--319, 2022.
\newblock \doi{10.2140/gt.2022.26.283}.

\bibitem{BG17}
A.~Broustet and Y.~Gongyo.
\newblock Remarks on log {C}alabi-{Y}au structure of varieties admitting polarized endomorphisms.
\newblock {\em Taiwanese J. Math.}, 21(3):569--582, 2017.
\newblock \doi{10.11650/tjm/7968}.

\bibitem{BH93}
W.~Bruns and J.~Herzog.
\newblock {\em Cohen-{M}acaulay rings}, volume~39 of {\em Cambridge Studies in Advanced Mathematics}.
\newblock Cambridge University Press, Cambridge, 1993.

\bibitem{CLS11}
D.~A. Cox, J.~B. Little, and H.~K. Schenck.
\newblock {\em Toric varieties}, volume 124 of {\em Graduate Studies in Mathematics}.
\newblock American Mathematical Society, Providence, RI, 2011.
\newblock \doi{10.1090/gsm/124}.

\bibitem{DethloffGrauert94}
G.~Dethloff and H.~Grauert.
\newblock Seminormal complex spaces.
\newblock In {\em Several complex variables, {VII}}, volume~74 of {\em Encyclopaedia Math. Sci.}, pages 183--220. Springer, Berlin, 1994.
\newblock \doi{10.1007/978-3-662-09873-8\_5}.

\bibitem{FM20}
S.~Filipazzi and J.~Moraga.
\newblock Strong {$(\delta,n)$}-complements for semi-stable morphisms.
\newblock {\em Doc. Math.}, 25:1953--1996, 2020.
\newblock \doi{10.4171/DM/790}.

\bibitem{GKP16arxiv}
D.~Greb, S.~Kebekus, and T.~Peternell.
\newblock Étale fundamental groups of kawamata log terminal spaces, flat sheaves, and quotients of abelian varieties, 2013, \burlalt{arXiv:1307.5718}{http://arxiv.org/abs/1307.5718}.

\bibitem{Hor17}
A.~H\"{o}ring.
\newblock Totally invariant divisors of endomorphisms of projective spaces.
\newblock {\em Manuscripta Math.}, 153(1-2):173--182, 2017.
\newblock \doi{10.1007/s00229-016-0881-8}.

\bibitem{Hu18}
F.~Hu.
\newblock The dimension of automorphism groups of algebraic varieties with pseudo-effective log canonical divisors.
\newblock {\em Proc. Amer. Math. Soc.}, 146(5):1879--1893, 2018.
\newblock \doi{10.1090/proc/13893}.

\bibitem{HN11}
J.-M. Hwang and N.~Nakayama.
\newblock On endomorphisms of {F}ano manifolds of {P}icard number one.
\newblock {\em Pure Appl. Math. Q.}, 7(4):1407--1426, 2011.
\newblock \doi{10.4310/PAMQ.2011.v7.n4.a15}.

\bibitem{KT23}
T.~Kawakami and B.~Totaro.
\newblock Endomorphisms of varieties and {B}ott vanishing, 2023, \burlalt{arXiv:2302.11921}{http://arxiv.org/abs/2302.11921}.

\bibitem{Kol13}
J.~Koll\'{a}r.
\newblock {\em Singularities of the minimal model program}, volume 200 of {\em Cambridge Tracts in Mathematics}.
\newblock Cambridge University Press, Cambridge, 2013.
\newblock \doi{10.1017/CBO9781139547895}.
\newblock With a collaboration of S\'{a}ndor Kov\'{a}cs.

\bibitem{KX09arxiv}
J.~Kollár and C.~Xu.
\newblock Fano varieties with large degree endomorphisms, 2009, \burlalt{arXiv:0901.1692}{http://arxiv.org/abs/0901.1692}.

\bibitem{LLM19}
A.~Laface, A.~Liendo, and J.~Moraga.
\newblock The fundamental group of a log terminal {$\Bbb T$}-variety.
\newblock {\em Eur. J. Math.}, 5(3):937--957, 2019.
\newblock \doi{10.1007/s40879-018-0296-z}.

\bibitem{Laz04a}
R.~Lazarsfeld.
\newblock {\em Positivity in algebraic geometry. {I}}, volume~48 of {\em Ergebnisse der Mathematik und ihrer Grenzgebiete. 3. Folge. A Series of Modern Surveys in Mathematics [Results in Mathematics and Related Areas. 3rd Series. A Series of Modern Surveys in Mathematics]}.
\newblock Springer-Verlag, Berlin, 2004.
\newblock \doi{10.1007/978-3-642-18808-4}.
\newblock Classical setting: line bundles and linear series.

\bibitem{Meng-17}
S.~Meng.
\newblock Building blocks of amplified endomorphisms of normal projective varieties.
\newblock {\em Math. Z.}, 294(3-4):1727--1747, 2020.
\newblock \doi{10.1007/s00209-019-02316-7}.

\bibitem{MZ19}
S.~Meng and D.-Q. Zhang.
\newblock Characterizations of toric varieties via polarized endomorphisms.
\newblock {\em Math. Z.}, 292(3-4):1223--1231, 2019.
\newblock \doi{10.1007/s00209-018-2160-8}.

\bibitem{MZ20}
S.~Meng and D.-Q. Zhang.
\newblock Normal projective varieties admitting polarized or int-amplified endomorphisms.
\newblock {\em Acta Math. Vietnam.}, 45(1):11--26, 2020.
\newblock \doi{10.1007/s40306-019-00333-6}.

\bibitem{MZZ22}
S.~Meng, D.-Q. Zhang, and G.~Zhong.
\newblock Non-isomorphic endomorphisms of {F}ano threefolds.
\newblock {\em Math. Ann.}, 383(3-4):1567--1596, 2022.
\newblock \doi{10.1007/s00208-021-02274-8}.

\bibitem{MZ23}
S.~Meng and G.~Zhong.
\newblock Rigidity of rationally connected smooth projective varieties from dynamical viewpoints.
\newblock {\em Math. Res. Lett.}, 30(2):589--610, 2023.
\newblock \doi{10.4310/MRL.2023.v30.n2.a10}.

\bibitem{Mor21}
J.~Moraga.
\newblock On a toroidalization for klt singularities, 2021, \burlalt{arXiv:2106.15019}{http://arxiv.org/abs/2106.15019}.

\bibitem{Mor22}
J.~Moraga.
\newblock Coregularity of {F}ano varieties, 2022, \burlalt{arXiv:2206.10834}{http://arxiv.org/abs/2206.10834}.

\bibitem{Nak02}
N.~Nakayama.
\newblock Ruled surfaces with non-trivial surjective endomorphisms.
\newblock {\em Kyushu J. Math.}, 56(2):433--446, 2002.
\newblock \doi{10.2206/kyushujm.56.433}.

\bibitem{PS89}
K.~H. Paranjape and V.~Srinivas.
\newblock Self-maps of homogeneous spaces.
\newblock {\em Invent. Math.}, 98(2):425--444, 1989.
\newblock \doi{10.1007/BF01388861}.

\bibitem{Saltman84}
D.~J. Saltman.
\newblock Noether's problem over an algebraically closed field.
\newblock {\em Invent. Math.}, 77(1):71--84, 1984.
\newblock \doi{10.1007/BF01389135}.

\bibitem{Wat81}
K.~Watanabe.
\newblock Some remarks concerning {D}emazure's construction of normal graded rings.
\newblock {\em Nagoya Math. J.}, 83:203--211, 1981.
\newblock \doi{10.1017/S0027763000019498}.

\bibitem{Yos21}
S.~Yoshikawa.
\newblock Structure of {F}ano fibrations of varieties admitting an int-amplified endomorphism.
\newblock {\em Adv. Math.}, 391:Paper No. 107964, 32, 2021.
\newblock \doi{10.1016/j.aim.2021.107964}.

\end{thebibliography}

\end{document}